\documentclass{amsart}
\usepackage[utf8]{inputenc}
\usepackage{a4wide}
\usepackage{amsmath, amssymb, bbm}
\usepackage{amsthm}
\usepackage{mathtools}
\usepackage{bbm}
\usepackage{blkarray}
\usepackage{tikz}
\usepackage{graphicx}
\usepackage{color}
\usepackage{enumerate}
\definecolor{darkblue}{rgb}{0,0,0.6}
\usepackage[ocgcolorlinks,colorlinks=true, citecolor=darkblue, filecolor=darkblue, linkcolor=darkblue, urlcolor=darkblue]{hyperref}
\usepackage[capitalize]{cleveref}
\usepackage{comment}
\usepackage{pifont}
\usepackage[mathscr]{euscript}

\newcounter{commentcounter}

\newtheoremstyle{thms}
	{}{}{\itshape}{}{\bfseries }{}{ }
	{\thmname{#1} \thmnumber{#2}. \thmnote{\bfseries{[#3]}}}
\newtheoremstyle{name}
	{}{}{\itshape}{}{\bfseries }{}{ }
	{\thmname{#1}\thmnumber{#2}\thmnote{\bfseries{[#3]}}}
\newtheoremstyle{defs}
	{}{12pt}{\normalfont}{}{\bfseries }{}{ }
	{\thmname{#1} \thmnumber{#2}. \thmnote{\bfseries{(#3)}}}
\newtheoremstyle{defs2}
	{}{12pt}{\normalfont}{}{\bfseries }{}{ }
	{\thmname{#1}\thmnumber{#2}. \thmnote{\bfseries{(#3)}}}
\newtheoremstyle{rmk}
	{}{}{\normalfont}{}{\itshape }{}{ }{\thmname{#1}. \thmnote{#3}}
\newtheoremstyle{claim}
	{}{}{\normalfont}{}{\itshape}{}{ }{\thmname{#1} \thmnumber{#2}. \thmnote{#3}}

\theoremstyle{thms}
\newtheorem{Prop}{Proposition}
\newtheorem{Thm}[Prop]{Theorem}
\newtheorem*{Thm*}{Theorem}

\newtheorem{Lemma}[Prop]{Lemma}
\newtheorem{Cor}[Prop]{Corollary}

\theoremstyle{name}

\theoremstyle{defs2}
\newtheorem*{ackn}{Acknowledgements}

\theoremstyle{defs}
\newtheorem{Def}[Prop]{Definition}

\theoremstyle{rmk}
\newtheorem{Rmk}{Remark}

\theoremstyle{rmk}

\theoremstyle{rmk}

\theoremstyle{claim}

\renewcommand{\c}{\mathscr{C}}
\newcommand{\D}{\mathscr{D}}
\newcommand{\Sp}{\mathrm{Sp}}

\newcommand{\Fun}{\mathrm{Fun}}
\newcommand{\Grp}{\mathrm{Grp}}

\newcommand{\id}{\mathrm{id}}
\newcommand{\wh}{\widehat}
\newcommand{\diag}{\mathrm{diag}}
\newcommand{\const}{\operatorname{const}}
\newcommand{\ev}{\operatorname{ev}}
\newcommand{\Gpd}{\mathrm{Gpd}}
\newcommand{\cc}{\mathrm{C^*Alg}}
\newcommand{\KK}{\mathrm{KK}}
\newcommand{\G}{\mathcal{G}}
\newcommand{\C}{\mathbb{C}}

\newcommand{\End}{\mathrm{End}}
\newcommand{\F}{\mathcal{F}}
\newcommand{\M}{\mathscr{M}}

\newcommand{\xto}{\xrightarrow}
\newcommand{\Orb}{\mathrm{Orb}}
\newcommand{\N}{\mathrm{N}}
\newcommand{\A}{\mathscr{A}}

\newcommand{\Reedy}{\mathrm{R}}
\newcommand{\we}{\mathrel\sim}
\newcommand{\weto}{\stackrel{\we}{\to}}

\newcommand{\cto}{\rightarrowtail}
\newcommand{\acto}{\stackrel{\we}{\cto}}
\newcommand{\Ncd}{\mathrm{N}^\mathrm{R}}
\renewcommand{\tilde}{\widetilde}
\renewcommand{\hat}{\widehat}
\newcommand{\cof}{c\:\!}

\newenvironment{tikzeq*}
{
  \begingroup
  \begin{equation*}
  \begin{tikzpicture}[baseline=(current bounding box.center)]
}
{
  \end{tikzpicture}
  \end{equation*}
  \endgroup
  \ignorespacesafterend
}

\usetikzlibrary{matrix,arrows,positioning}

\tikzset
{
  diagram/.style=
  {
    matrix of math nodes,
    column sep=2.5em,
    row sep=2.5em,
    text height=1.5ex,
    text depth=.25ex
  },
  over/.style={preaction={draw=white,-,line width=6pt}},
  every to/.style={font=\footnotesize},
  cof/.style={>->},                
  fib/.style={->>},                
  inj/.style={right hook->},       
  surj/.style={-open triangle 60}, 
  eq/.style={-,double distance=1.7pt}
}

\begin{document}

\title{Localization of Cofibration Categories and Groupoid $C^*$-algebras}
\date{\today}

\author[M.~Land]{Markus Land}
\address{University of Bonn, Mathematical Institute, Endenicher Allee 60, 53115 Bonn}
\email{land@math.uni-bonn.de}

\author[T.~Nikolaus]{Thomas Nikolaus}
\address{Max Planck Institut f\"ur Mathematik Bonn, Vivatsgasse 7, 53111 Bonn}
\email{thoni@mpim-bonn.mpg.de}

\author[K.~Szumi{\l}o]{Karol Szumi{\l}o}
\address{University of Western Ontario, Department of Mathematics, 1151 Richmond Street, N6A 3K7 London, ON}
\email{kszumilo@uwo.ca}

\begin{abstract}
We prove that relative functors out of a cofibration category are essentially the same as relative functors which are only defined on the subcategory of cofibrations. As an application we give a new construction of the functor that assigns to a groupoid its groupoid $C^*$-algebra and thereby its topological $K$-theory spectrum.
\end{abstract}

\maketitle

Let $(\c,w\c,\cof\c)$ be a \emph{cofibration category},
i.e.\ a structure dual to a category of fibrant objects in the sense of Brown \cite{Brown}.
Here, $w\c$ and $\cof\c$ are the subcategories of weak equivalences and cofibrations,
i.e.\ they have the same objects as $\c$ but morphisms are the weak equivalences
or the cofibrations respectively.
Similarly, $w\cof\c$ will denote the subcategory of acyclic cofibrations.
In~addition to Brown's axioms, we will assume that $\c$ has \emph{good cylinders}
which is a mild technical condition explained in \cref{good-cylinders}.
In this paper we will prove the following theorem.

\begin{Thm}\label{main theorem}
If a cofibration category $\c$ has good cylinders, then the map induced by the inclusion
\begin{tikzeq*}
  \matrix[diagram]
  {
    |(cC)| \N \cof\c[w\cof^{-1}] & |(C)| \N\c[w^{-1}] \\
  };

  \draw[->] (cC) to node[above] {$\simeq$} (C);
\end{tikzeq*}
is an equivalence of $\infty$-categories. In particular, by passing to homotopy categories, we obtain an equivalence of ordinary categories $\cof\c[w\cof^{-1}] \xto{\simeq} \c[w^{-1}] $.
\end{Thm}
By $\N\c[w^{-1}]$ we denote the universal $\infty$-category obtained from $\N\c$ by inverting the weak equivalences, see \cite[Def. 13.4.1 and Remark 1.3.4.2]{LurieHA}. The same universal property in the world of ordinary categories describes $\c[w^{-1}]$. By passing to opposite categories, the dual statement of \cref{main theorem} for fibration categories also holds.

The proof of \cref{main theorem} will be given at the end of the paper, but let us first establish a consequence and the application to $C^*$-algebras associated to groupoids.

Let $\c$ be a small cofibration category with good cylinders and
$\M$ a model category which is Quillen equivalent to a combinatorial model category
and has functorial fibrant and cofibrant replacements, e.g.\ any of the model categories of spectra.
\begin{Prop}\label{prop}
For any functor $F\colon \cof\c \to \M$ that sends acyclic cofibrations in $\cof\c$ to weak equivalences in $\M$ there exists a functor $\wh{F}\colon \c \to \M$ with the following properties:
\begin{enumerate}
\item\label{eins} $\wh{F}$ sends weak equivalences in $\c$ to weak equivalences in $\M$. 
\item\label{zwei} $\wh{F}$ extends $F$ in the sense that there exists a zig-zag of natural weak equivalences between $F$ and $\wh{F}\mid_{\cof\c}$.
\end{enumerate}
Moreover $\wh{F}$ is unique in the following sense: for any other functor $\wh{F}': \c \to \M$ that satisfies \eqref{eins} and \eqref{zwei} there exists a zig-zag of natural weak equivalences between
$\wh{F}$ and $\wh{F}'$.
\end{Prop}
\begin{proof}
We denote the $\infty$-category 
$\N\M[w^{-1}]$  associated to the model category $\M$ by $\M_\infty$. We claim that for any ordinary category $\A$ the canonical map
\[ \N\Fun(\A,\M)[\ell^{-1}] \to \Fun(\N\A, \M_\infty) \]
is an equivalence of $\infty$-categories, where $\ell$ is the class of levelwise weak equivalences. If $\M$ is a simplicial, combinatorial model category this is a special case of 
\cite[Proposition 4.2.4.4]{LurieHTT} using that for a simplicial model category $\M$, the $\infty$-category $\M_\infty$ is equivalent to  the homotopy coherent nerve of the simplicial subcategory of $\M$ on the fibrant and cofibrant objects, see \cite[Theorem 1.3.4.20]{LurieHA}.
From the existence of functorial (co)fibrant replacements and \cite[Proposition 1.3.13]{Hovey} it follows that a Quillen equivalence $\M \simeq \M'$ induces a Quillen equivalence $\Fun(\A,\M) \simeq \Fun(\A,\M')$. Thus the domain of the map in question is invariant under Quillen equivalences in $\M$.
The same is true for the codomain, thus the statement that this map is an equivalence is invariant under Quillen equivalences in $\M$.
Hence it is also true for all model categories $\M$ with functorial (co)fibrant replacements that are Quillen equivalent to a combinatorial, simplicial model category. 
Since every combinatorial model category is equivalent to a combinatorial, simplicial model category by a result of Dugger \cite[Corollary 1.2]{MR1870516}, the claim holds in our generality. If $\A$ is a relative category it also follows that the induced functor
\[ \N\Fun^w(\A,\M)[\ell^{-1}] \to \Fun^w(\N\A, \M_\infty) \]
is an equivalence, where the superscript $w$ refers to functors that send weak equivalences in $\A$ to weak equivalences, respectively equivalences in the target.
Thus in the canonical commuting square
\begin{tikzeq*}
  \matrix[diagram]
  {
    |(m)| \N\Fun^w(\c, \M)[\ell^{-1}] & |(mc)| \N\Fun^w(\cof\c,\M)[\ell^{-1}] \\
    |(i)| \Fun^w(\N\c, \M_\infty)  & |(ic)| \Fun^w(\N\cof\c,\M_\infty)   \\
  };

  \draw[->] (m)  to (i);
  \draw[->] (mc) to (ic);
  \draw[->] (m)  to (mc);
  \draw[->] (i)  to (ic);
\end{tikzeq*}
the vertical maps are equivalences of $\infty$-categories. By  \cref{main theorem} the lower map is also an equivalence, therefore also the upper one is. 
Passing to homotopy categories we obtain the desired result, using that isomorphisms in homotopy categories of functor categories are represented by zig-zags of natural weak equivalences.
\end{proof}

\section*{Applications}

\subsection*{Groupoids}
We denote by $\Gpd$ the $1$-category of small groupoids and by $\Gpd_2$ the $\infty$-category associated to the $(2,1)$-category of groupoids in which the $2$-morphisms are natural transformations. The category $\Gpd$ admits a simplicial model structure in which the equivalences are equivalences of categories and the cofibrations are functors that are injective on the set of objects. In this model structure all objects are cofibrant and fibrant, compare \cite{Casacuberta}. Furthermore if we denote by $\Gpd^\omega$ the full subcategory on groupoids with at most countable many morphisms then $\Gpd^\omega$ inherits the structure of a cofibration category.

The following lemma is a well known fact, but we had difficulties finding a clear reference for this so we state it as an extra lemma.
\begin{Lemma}\label{lemma3}
The canonical map $\N\Gpd[w^{-1}] \to \Gpd_2$ is an equivalence of $\infty$-categories.
\end{Lemma}
\begin{proof}
This follows from the description of the $\infty$-category associated to a simplicial model category, see \cite[Theorem 1.3.4.20]{LurieHA}, as being the the homotopy coherent nerve of the simplicial category of cofibrant and fibrant objects.
\end{proof}

\begin{Cor}\label{groupoids1}
Let $\c$ be an $\infty$-category. Then the canonical map $\N \cof\Gpd \to \Grp_2 $ induces an equivalence
\begin{tikzeq*}
  \matrix[diagram]
  {
    |(2)| \Fun(\Gpd_2,\c) & |(i)| \Fun^w(\N \cof\Gpd,\c) \\
  };

  \draw[->] (2) to node[above] {$\simeq$} (i);
\end{tikzeq*}
where the superscript $w$ refers to functors that send equivalences of groupoids to equivalences in $\c$.
\end{Cor}
\begin{proof}
Since the canonical map  $\N\Gpd[w^{-1}] \to \Gpd_2$ is an equivalence by \cref{lemma3}, this is a direct application of \cref{main theorem}.
\end{proof}

The following corollary of Proposition \ref{prop} implies that in the approach to assembly maps discussed in \cite[section 2]{DavisLueck} one can directly restrict to functors from groupoids to spectra that are only defined for maps of groupoids that are injective on objects. This resolves the issues illustrated in \cite[Remark 2.3]{DavisLueck}.
\begin{Cor}\label{groupoidscor2}
Let $\Sp$ be any of the categories of spectra. Then every functor $F: \cof\Gpd \to \Sp$ which sends equivalences of groupoids to weak equivalences in $\Sp$ extends uniquely (in the sense of Proposition \ref{prop}) to a functor 
$\widehat{F}: \Gpd \to \Sp$ which also sends weak equivalences of groupoids to weak equivalences of spectra.
\end{Cor}
\begin{Rmk}
The statements of \cref{groupoids1} and \cref{groupoidscor2} remain  true if we replace $\Gpd$ by $\Gpd^\omega$. Furthermore \cref{groupoidscor2} does not depend on the exact choice of model category of spectra as long as it is Quillen equivalent to a combinatorial model category. 
\end{Rmk}

Next we want to demonstrate how to apply these results by \emph{functorially} constructing $C^*$-algebras and topological $K$-theory spectra associated to groupoids. This discussion is similar to the one given in \cite[section 3]{Joachim2} but we use our main theorem to obtain full functoriality instead of an explicit construction.
\begin{Def}\label{groupoid algebra}
Let $\G$ be a groupoid.
We let $\C\G$ be the $\C$-linearization of the set of morphisms of $\G$. This is a $\C$-algebra by linearization of the multiplication on morphisms given by
\[ f\cdot g = \begin{cases} f\circ g & \text{ if $f$ and $g$ are composable} \\ 0 & \text{ else.} \end{cases} \] 
We remark that $\C\G$ is unital if and only if the set of objects of $\G$ is finite.
Then we complete $\C\G$ in a universal way, like for the full group $C^*$-algebra, to obtain a $C^*$-algebra $C^*\G$. More precisely, the norm is given by the supremum over all norms of representations of $\C\G$ on a separable Hilbert space.
This is isomorphic to the $C^*$-algebra associated to the maximal groupoid $C^*$-category of \cite[Definition 3.16]{DellAmbrogio2} using the construction $\mathcal{C} \mapsto A_\mathcal{C}$ of \cite[section 3]{Joachim2}. 
\end{Def}
The association $\G \mapsto \C^*\G$ is functorial for cofibrations of groupoids but not for general morphisms since it can happen that morphisms are not composable in a groupoid, but become composable after applying a functor, compare the remark \cite[page 214]{DavisLueck}. We observe that the $C^*$-algebra $\C^*\G$ is separable provided $\G \in \Gpd^\omega$. 

\begin{Lemma}\label{morita invariance}
Let $F\colon \G_1 \to \G_2$ be an acyclic cofibration of groupoids. Then the induced morphism
\[ C^*F \colon C^*\G_1 \to C^*\G_2 \]
is a $\KK$-equivalence.
\end{Lemma}
\begin{proof}
The $C^*$-algebra associated to a groupoid is the product of the $C^*$-algebras associated to each connected component. Thus we may assume that $\G_1$ (and thus $\G_2$) is connected.
Let $x \in \G$ be an object. We let $G_1 = \End(x)$ and $G_2 = \End(Fx)$ be the endomorphism groups and notice the fact that $F$ is an equivalence implies that $F$ induces an isomorphism $G_1 \cong G_2$.
Then we consider the diagram
\begin{tikzeq*}
  \matrix[diagram]
  {
    |(1c)| C^*\G_1 & |(2c)| C^*\G_2 \\
    |(1)|  C^*G_1  & |(2)|  C^*G_2  \\
  };

  \draw[->] (1c) to node[above] {$C^*F$} (2c);
  \draw[->] (1)  to node[below] {$\cong$} (2);

  \draw[->] (1) to (1c);
  \draw[->] (2) to (2c);
\end{tikzeq*}
in which the lower horizontal arrow is an isomorphism. Thus to show the lemma it suffices to prove the lemma in the special case where $F$ is the inclusion of the endomorphisms of an object $x$ of a connected groupoid $\G$.

This can be done using in the abstract setting of corner algebras. For this suppose $A$ is a $C^*$-algebra and $p \in A$ is a projection. It is called full if $ApA$ is dense in $A$. The algebra $pAp$ is called the corner algebra of $p$ in $A$. It is called a full corner if $p$ is a full projection. We write $i_p$ for the inclusion $pAp \subset A$. Given a projection $p$ the module $pA$ is an imprimitivity $pAp-\overline{ApA}$ bimodule, see e.g. \cite[Example 3.6]{Raeburn}. Thus if $p$ is full, then $pA$ gives rise to an invertible element $[pA,i_p,0]= \F(p) \in \KK(pAp,A)$.  In this KK-group we have an equality
\begin{align*} \F(p) & = [\,pA,i_p,0] + [(1-p)A,0,0] \\ & = [\,pA\oplus(1-p)A,i_p,0] \\ &= [A,i_p,0] = [i_p] ,
\end{align*}
in other words, the inclusion $pAp \to A$ of a corner algebra associated to a full projection is a KK-equivalence. 

To come back to our situation let us suppose $\G$ is a groupoid, $x \in \G$ is an object and let us denote its endomorphism group by $G = \End(x)$. We can consider the element $p = \id_x \in C^*\G$ which is clearly a projection. Its corner algebra is given by
\[ p\cdot C^*\G \cdot p \cong C^*G.\]
If $\G$ is connected it follows that every morphism in $\G$ may be factored through $\id_x$ and thus $p$ is full if $\G$ is connected. Hence it follows that the inclusion $C^*G \to C^*\G$ is an embedding of a full corner algebra. Thus by the general theory this inclusion is a KK-equivalence which proves the lemma.
\end{proof}

Let us denote by $\KK_\infty$ the $\infty$-category given by the localization of the category $\cc$ of separable $C^*$-algebras at the KK-equivalences, see e.g. \cite[Definition 3.2]{LandNikolaus}. In formulas we have $\KK_\infty := \N\cc[w^{-1}]$ where $w$ denotes the class of KK-equivalences. The homotopy category of $\KK_\infty$ is Kasparov's KK-category of $C^*$-algebras.

\begin{Cor}\label{groupoids2}
There exists a functor
\[ \Gpd_2^\omega \to \KK_\infty \]
which on objects sends a groupoid $\G$ to the full groupoid $C^*$-algebra $C^*\G$.
\end{Cor}
\begin{Rmk}
We notice that the $(2,1)$-category $\Orb^\omega$ consisting of (countable) groups, group homomorphisms, and conjugations is the full subcategory of the $(2,1)$-category of (countable) groupoids on connected groupoids and hence along this inclusion we also obtain a functor
\[ \Orb^\omega \to \KK_\infty \]
which on objects sends a group to its full group $C^*$-algebra. This will be used in \cite{LandNikolaus} to compare the $L$-theoretic Farrell-Jones conjecture and the Baum-Connes conjecture. 
\end{Rmk}
\begin{proof}[Proof of \cref{groupoids2}]
By \cref{groupoids1} and the remark after \cref{groupoidscor2}, we have an equivalence
\[ \Fun^w(\N \cof\Gpd^\omega,\KK_\infty) \simeq \Fun(\Gpd_2^\omega,\KK_\infty) \]
and thus it suffices to construct a functor
\[ \cof\Gpd^\omega \to \cc \]
which has the \emph{property} that it sends equivalences of groupoids to $\KK$-equivalences.
We have established in \cref{morita invariance} that the functor of \cref{groupoid algebra} satisfies this property.
\end{proof}

\begin{Rmk}
In \cite[Proposition 3.7]{LandNikolaus} it is shown that the topological $K$-theory functor 
\[ K\colon \N\cc \to \Sp \]
factors over $\KK_\infty$, in fact becomes corepresentable there. It thus follows from \cref{groupoids2} that there is a functor sending a groupoid to the topological $K$-theory spectrum of its $C^*$-algebra.
\end{Rmk}

\section*{The proof of \cref{main theorem}}

In this section we will prove \cref{main theorem}.
Recall that we consider a cofibration category $(\c,w\c,\cof\c)$ and
aim to compare the $\infty$-categories associated to
the relative categories $(\c,w\c)$ and $(\cof\c,w\cof\c)$.
As our model of the homotopy theory of $(\infty,1)$-categories we will use
\emph{complete Segal spaces} of Rezk, see \cite{Rezk}.
This homotopy theory is modelled by the Rezk model structure on the category of
bisimplicial sets in which fibrant objects are the complete Segal spaces.
The model structure is constructed as a Bousfield localization of
the Reedy model structure and hence every levelwise weak equivalence of
bisimplicial sets is a Rezk equivalence, i.e.\ an equivalence of $\infty$-categories.

The $\infty$-category associated to a relative category $(\D,w\D)$ is modelled by
the \emph{classification diagram} $\Ncd\D$ of Rezk which is given by
\[ (\Ncd\D)_k \mapsto \N w(\D^{[k]}),\]
where the weak equivalences in $\D^{[k]}$ are levelwise weak equivalences, 
compare \cite[section 3.3]{Rezk} and \cite[Theorem 3.8]{Mazel-Gee}. See also the MathOverflow post \cite{Cisinksi2}.
The classification diagram is not fibrant in the Rezk model structure,
but it is levelwise equivalent to a fibrant object if $\D$ is a cofibration category.

Recall that we stated \cref{main theorem} under the following assumption on the cofibration category $\c$.
\begin{Def}\label{good-cylinders}
  A cofibration category $\c$ has \emph{good cylinders} if
  it has a cylinder functor $I$ such that for every cofibration $X \cto Y$
  the induced morphism $I X \sqcup_{X \sqcup X} (Y \sqcup Y) \to I Y$
  is a cofibration.
\end{Def}

For example any cofibration category arising from a monoidal model category
(or a model category enriched over a monoidal model category)
has good cylinders, since
they are given by tensoring with a chosen interval object.

\begin{Thm}\label{cC-C}
  If $\c$ has good cylinders, then
  the inclusion $\cof\c \to \c$ induces a levelwise weak equivalence of
  the classification diagrams $\Ncd \cof\c \to \Ncd\c$.
\end{Thm}

For the proof we will need a series of auxiliary definitions and lemmas.
Let us first fix some notation.
If $J$ is a category, then $\hat{J}$ denotes $J$ considered as
a relative category with all morphisms as weak equivalences.
If $J$ is any relative category, then $\c^J$ stands for
the cofibration category of all relative diagrams $J \to \c$ with
levelwise weak equivalences and cofibrations.
If $J$ is any relative direct category, then $\c^J_\Reedy$ stands for
the cofibration category of
all relative Reedy cofibrant diagrams $J \to \c$ with
levelwise weak equivalences and Reedy cofibrations.
See \cite[Theorem 9.3.8]{Radulescu-Banu}
for the construction of these cofibration categories.

\begin{Def}
  A subcategory $g\c$ of a cofibration category $\c$ is said to be \emph{good} if
  \begin{itemize}
  \item all cofibrations are in $g\c$;
  \item the morphisms of $g\c$ are
    stable under pushouts along cofibrations;
  \item $\c$ has functorial factorizations that preserve $g\c$
    in the sense that if
    \begin{tikzeq*}
      \matrix[diagram]
      {
        |(A0)| A_0 & |(B0)| B_0 \\
        |(A1)| A_1 & |(B1)| B_1 \\
      };

      \draw[->] (A0) to (B0);
      \draw[->] (A1) to (B1);

      \draw[->] (A0) to (A1);
      \draw[->] (B0) to (B1);
    \end{tikzeq*}
    is a square in $\c$ such that both vertical morphisms are in $g\c$
    and
    \begin{tikzeq*}
      \matrix[diagram]
      {
        |(A0)| A_0 & |(tB0)| \tilde{B}_0 & |(B0)| B_0 \\
        |(A1)| A_1 & |(tB1)| \tilde{B}_1 & |(B1)| B_1 \\
      };

      \draw[cof] (A0) to (tB0);
      \draw[cof] (A1) to (tB1);

      \draw[->] (tB0) to node[above] {$\we$} (B0);
      \draw[->] (tB1) to node[below] {$\we$} (B1);

      \draw[->] (A0) to (A1);
      \draw[->] (tB0) to (tB1);
      \draw[->] (B0) to (B1);
    \end{tikzeq*}
    is the resulting factorization, then the induced morphism
    $A_1 \sqcup_{A_0} \tilde{B}_0 \to \tilde{B}_1$ is also in $g\c$.
    (In particular, so is $\tilde{B}_0 \to \tilde{B}_1$
    by the second condition.)
  \end{itemize}
\end{Def}

Now suppose that $\c$ is cofibration category with a good subcategory $g\c$.
We let $W\c$ be the bisimplicial set whose $(m,n)$-bisimplices are
all diagrams in $\c$ of the form
\begin{tikzeq*}
  \matrix[diagram]
  {
    |(00)| X_{0,0} & |(01)| X_{0,1} & |(0l)| \ldots & |(0n)| X_{0,n} \\
    |(10)| X_{1,0} & |(11)| X_{1,1} & |(1l)| \ldots & |(1n)| X_{1,n} \\
    |(l0)| \vdots  & |(l1)| \vdots  &               & |(ln)| \vdots  \\
    |(m0)| X_{m,0} & |(m1)| X_{m,1} & |(ml)| \ldots & |(mn)| X_{m,n}
    \text{,} \\
  };

  \draw[cof] (00) to node[above] {$\we$} (01);
  \draw[cof] (01) to node[above] {$\we$} (0l);
  \draw[cof] (0l) to node[above] {$\we$} (0n);

  \draw[cof] (10) to node[above] {$\we$} (11);
  \draw[cof] (11) to node[above] {$\we$} (1l);
  \draw[cof] (1l) to node[above] {$\we$} (1n);

  \draw[cof] (m0) to node[below] {$\we$} (m1);
  \draw[cof] (m1) to node[below] {$\we$} (ml);
  \draw[cof] (ml) to node[below] {$\we$} (mn);

  \draw[->] (00) to node[left] {$\we$} node[below right = 0.25 cm and 0.02 cm] {\tiny$g$} (10);
  \draw[->] (10) to node[left] {$\we$} node[below right = 0.25 cm and 0.02 cm] {\tiny$g$} (l0);
  \draw[->] (l0) to node[left] {$\we$} node[below right = 0.25 cm and 0.02 cm] {\tiny$g$} (m0);

  \draw[->] (01) to node[left] {$\we$} node[below right = 0.25 cm and 0.02 cm] {\tiny$g$} (11);
  \draw[->] (11) to node[left] {$\we$} node[below right = 0.25 cm and 0.02 cm] {\tiny$g$} (l1);
  \draw[->] (l1) to node[left] {$\we$} node[below right = 0.25 cm and 0.02 cm] {\tiny$g$} (m1);

  \draw[->] (0n) to node[right] {$\we$} node[below right = 0.25 cm and 0.02 cm] {\tiny$g$} (1n);
  \draw[->] (1n) to node[right] {$\we$} node[below right = 0.25 cm and 0.02 cm] {\tiny$g$} (ln);
  \draw[->] (ln) to node[right] {$\we$} node[below right = 0.25 cm and 0.02 cm] {\tiny$g$} (mn);
\end{tikzeq*}
i.e.\ diagrams $\hat{[m]} \times \hat{[n]} \to \c$ where
all horizontal morphisms are cofibrations and
all vertical morphisms are in $g\c$.
In other words $W\c$ is the nerve of a double category with
the same objects as $\c$, whose
horizontal morphisms are acyclic cofibrations,
vertical morphisms are weak equivalences in $g\c$,
and double morphisms are just commutative squares.

\begin{Lemma}
  The bisimplicial set $W\c$ is vertically homotopically constant,
  i.e.\ every simplicial operator $[n] \to [n']$ induces
  a weak homotopy equivalence $(W\c)_{*,n'} \to (W\c)_{*,n}$.
\end{Lemma}

\begin{proof}
  Note that $(W\c)_{*,n} = \N \tilde{\c}_n$ where
  $\tilde{\c}_n$ is a category whose
  objects are diagrams $\hat{[n]} \to \cof\c$ and
  whose morphisms are weak equivalences with all components in $g\c$.
  It is enough to consider the case $n' = 0$,
  i.e.\ to show that the constant functor
  $\const \colon \tilde{\c}_0 \to \tilde{\c}_n$
  is a homotopy equivalence.
  The evaluation at $n$ functor
  $\ev_n \colon \tilde{\c}_n \to \tilde{\c}_0$
  satisfies $\ev_n \const = \id_{\tilde{\c}_0}$.
  Moreover, the structure maps of
  every diagram $X \in \tilde{\c}_n$ form
  a natural weak equivalence $X \to \const \ev_n X$
  since every cofibration is in $g\c$.
\end{proof}

\begin{Lemma}\label{W-horizontal}
  The bisimplicial set $W\c$ is horizontally homotopically constant,
  i.e.\ every simplicial operator $[m] \to [m']$ induces
  a weak homotopy equivalence $(W\c)_{m',*} \to (W\c)_{m,*}$.
\end{Lemma}

\begin{proof}    
  Note that $(W\c)_{m,*} = \N \bar{\c}_m$ where
  $\bar{\c}_m$ is a category whose
  objects are diagrams $\hat{[m]} \to g\c$ and
  whose morphisms are acyclic levelwise cofibrations.
  Again, it is enough to consider the case $m' = 0$
  and to show that the constant functor
  $\const \colon \bar{\c}_0 \to \bar{\c}_m$
  and the evaluation at $m$ functor
  $\ev_m \colon \bar{\c}_n \to \bar{\c}_0$
  form a homotopy equivalence.

  We have $\ev_m \const = \id_{\bar{\c}_0}$.
  Moreover, given any object $X \in \bar{\c}_m$
  and $i \in [m]$ we consider the composite weak equivalence $X_i \weto X_m$.
  We combine it with the identity $X_m \to X_m$ and
  factor functorially the resulting morphism $X_i \sqcup X_m \to X_m$ as
  $X_i \sqcup X_m \cto \tilde{X}_i \weto X_m$.
  In the square
  \begin{tikzeq*}
    \matrix[diagram]
    {
      |(mi)| X_m \sqcup X_i     & |(i)| X_m \\
      |(m1)| X_m \sqcup X_{i+1} & |(1)| X_m \\
    };

    \draw[->] (mi) to (i);
    \draw[->] (m1) to (1);
    \draw[->] (i)  to (1);
    \draw[->] (mi) to (m1);
  \end{tikzeq*}
  both vertical morphisms are in $g\c$
  (since $g\c$ is closed under pushouts).
  Thus the induced morphism $\tilde{X}_i \to \tilde{X}_{i+1}$
  is in $g\c$.
  Moreover, we obtain acyclic cofibrations
  $X_i \acto \tilde{X}_i$ and $X_m \acto \tilde{X}_i$ that
  constitute a zig-zag of natural weak equivalences connecting
  $\const \ev_m$ and $\id_{\bar{\c}_m}$.
\end{proof}

\begin{Lemma}\label{W-diagonal}
  The inclusion $\N w\cof\c \to \N wg\c$ is a weak homotopy equivalence.
\end{Lemma}

\begin{proof}
  Observe that the $0$th row and the $0$th column of $W\c$ are
  $\N wg\c$ and $\N w\cof\c$ respectively.
  Since $W\c$ is homotopically constant in both directions,
  it follows from \cite[Proposition IV.1.7]{Goerss-Jardine}
  that we have weak equivalences
  \begin{tikzeq*}
    \matrix[diagram]
    {
      |(g)| \N wg\c & |(d)| \diag W\c & |(c)| \N w\cof\c \text{.} \\
    };

    \draw[->] (g) to node[above] {$\we$} (d);
    \draw[->] (c) to node[above] {$\we$} (d);
  \end{tikzeq*}
  Moreover, the restrictions along
  the diagonal inclusions $[m] \to [m] \times [m]$ induce
  a simplicial map $\diag W\c \to \N wg\c$ whose composites with
  the two maps above are the identity on $\N wg\c$ and
  the inclusion $\N w\cof\c \to \N wg\c$.
  Hence the latter is a weak equivalence by 2-out-of-3.
\end{proof}

Next we establish that under specific circumstances certain subcategories of $\c$ are good.

\begin{Lemma}\label{levelwise-good}
  Let $\c$ be a cofibration category.
  \begin{enumerate}
  \item If $\c$ has functorial factorizations,
    then $\c$ itself is a good subcategory.
  \item If $\c$ has good cylinders,
    then $\cof\c$ is a good subcategory of $\c$.
  \item If $\cof\c$ is a good subcategory of $\c$,
    then the subcategory of levelwise cofibrations is
    a good subcategory of $\c^{[k]}_\Reedy$ for all $k$.
  \end{enumerate}
\end{Lemma}

\begin{proof}
  \leavevmode
  \begin{enumerate}
  \item This is vacuously true.
  \item We will show that the standard mapping cylinder factorization
    makes $\cof\c$ into a good subcategory.
    Let
    \begin{tikzeq*}
      \matrix[diagram]
      {
        |(A0)| A_0 & |(B0)| B_0 \\
        |(A1)| A_1 & |(B1)| B_1 \\
      };

      \draw[->] (A0) to (B0);
      \draw[->] (A1) to (B1);

      \draw[cof] (A0) to (A1);
      \draw[cof] (B0) to (B1);
    \end{tikzeq*}
    be a square were both vertical morphisms are cofibrations.
    The mapping cylinder of $A_i \to B_i$ is constructed as
    $I A_i \sqcup_{A_i \sqcup A_i} (A_i \sqcup B_i)$.
    We need to show that the morphism induced by the square
    \begin{tikzeq*}
      \matrix[diagram]
      {
        |(A0)| A_0 & |(I0)| I A_0 \sqcup_{A_0 \sqcup A_0} (A_0 \sqcup B_0) \\
        |(A1)| A_1 & |(I1)| I A_1 \sqcup_{A_1 \sqcup A_1} (A_1 \sqcup B_1) \\
      };

      \draw[->] (A0) to (I0);
      \draw[->] (A1) to (I1);

      \draw[->] (I0) to (I1);
      \draw[->] (A0) to (A1);
    \end{tikzeq*}
    is a cofibration.
    This morphism coincides with
    \begin{tikzeq*}
      \matrix[diagram]
      {
        |(0)| I A_0 \sqcup_{A_0 \sqcup A_0} (A_1 \sqcup B_0) &
        |(1)| I A_1 \sqcup_{A_1 \sqcup A_1} (A_1 \sqcup B_1) \\
      };

      \draw[->] (0) to (1);
    \end{tikzeq*}
    which factors as
    \begin{tikzeq*}
      \matrix[diagram]
      {
        |(0)| I A_0 \sqcup_{A_0 \sqcup A_0} (A_1 \sqcup B_0) &
        |(1)| I A_0 \sqcup_{A_0 \sqcup A_0} (A_1 \sqcup B_1) &
        |(2)| I A_1 \sqcup_{A_1 \sqcup A_1} (A_1 \sqcup B_1) \text{.} \\
      };

      \draw[->] (0) to (1);
      \draw[->] (1) to (2);
    \end{tikzeq*}
    The first morphism is a pushout of $A_1 \sqcup B_0 \to A_1 \sqcup B_1$
    which is a cofibration since $B_0 \to B_1$ is.
    The second morphism is a pushout of
    $I A_0 \sqcup_{A_0 \sqcup A_0} (A_1 \sqcup A_1) \to I A_1$
    which is a cofibration since $A_0 \to A_1$ is and $\c$ has good cylinders.
  \item Clearly, every Reedy cofibration is a levelwise cofibration and
    levelwise cofibrations are stable under pullbacks.
    Consider a diagram
    \begin{tikzeq*}
      \matrix[diagram]
      {
        |(A0)| A_0 & |(tB0)| \tilde{B}_0 & |(B0)| B_0 \\
        |(A1)| A_1 & |(tB1)| \tilde{B}_1 & |(B1)| B_1 \\
      };

      \draw[cof] (A0) to (tB0);
      \draw[cof] (A1) to (tB1);

      \draw[->] (tB0) to node[above] {$\we$} (B0);
      \draw[->] (tB1) to node[below] {$\we$} (B1);

      \draw[->] (A0) to (A1);
      \draw[->] (tB0) to (tB1);
      \draw[->] (B0) to (B1);
    \end{tikzeq*}
    in $\c^J_\Reedy$ where $\tilde{B}_0$ and $\tilde{B}_1$
    are obtained by the standard Reedy factorization induced by
    the given functorial factorization in $\c$.
    Assuming that $A_0 \to A_1$ and $B_0 \to B_1$ are levelwise cofibrations,
    we need to check that
    $A_{1,i} \sqcup_{A_{0,i}} \tilde{B}_{0,i} \to \tilde{B}_{1,i}$
    is a cofibration for every $i \in [m]$.

    For $i = 0$, this follows directly from the assumption that
    $\cof\c$ is a good subcategory of $\c$.
    The Reedy factorization is constructed by induction over $[m]$,
    so assume that the conclusion is already known for $i < m$.
    The factorization at level $i+1$ arises as
    \begin{tikzeq*}
      \matrix[diagram]
      {
        |(A0)|  A_{0,i+1} \sqcup_{A_{0,i}} \tilde{B}_{0,i} &
        |(tB0)| \tilde{B}_{0,i+1} & |(B0)| B_{0,i+1} \\
        |(A1)|  A_{1,i+1} \sqcup_{A_{1,i}} \tilde{B}_{1,i} &
        |(tB1)| \tilde{B}_{1,i+1} & |(B1)| B_{1,i+1} \\
      };

      \draw[cof] (A0) to (tB0);
      \draw[cof] (A1) to (tB1);

      \draw[->] (tB0) to node[above] {$\we$} (B0);
      \draw[->] (tB1) to node[below] {$\we$} (B1);

      \draw[->] (A0) to (A1);
      \draw[->] (tB0) to (tB1);
      \draw[->] (B0) to (B1);
    \end{tikzeq*}
    where the left square comes from the diagram
    \begin{tikzeq*}
      \matrix[diagram,row sep=2em,column sep=2em]
      {
          |(A0i)| A_{0,i}   & & |(B0i)| \tilde{B}_{0,i} &[1em] \\
        & |(A01)| A_{0,i+1} & & |(P0)|  \bullet & & |(B01)| \tilde{B}_{0,i+1} \\
          |(A1i)| A_{1,i}   & & |(B1i)| \tilde{B}_{1,i} & \\
        & |(A11)| A_{1,i+1} & & |(P1)|  \bullet & & |(B11)| \tilde{B}_{1,i+1} \\
      };

      \draw[->] (B01) to (B11);
      \draw[->] (P0)  to (B01);
      \draw[->] (P1)  to (B11);

      \draw[->] (A0i) to (B0i);
      \draw[->] (A1i) to (B1i);

      \draw[->] (A0i) to (A1i);
      \draw[->] (B0i) to (B1i);

      \draw[->,over] (A01) to (P0);
      \draw[->] (A11) to (P1);

      \draw[->,over] (A01) to (A11);
      \draw[->] (P0) to (P1);

      \draw[->] (A0i) to (A01);
      \draw[->] (B0i) to (P0);
      \draw[->] (A1i) to (A11);
      \draw[->] (B1i) to (P1);
    \end{tikzeq*}
  where the bullets stand for the pushouts above.
  The conclusion we need to obtain amounts to the composite of
  the two squares in the front being a Reedy cofibration when seen as a morphism
  from left to right.
  The right square is a Reedy cofibration since
  $\cof\c$ is a good subcategory of $\c$ and
  so is the left one since it is a pushout of the back square
  which is a Reedy cofibration by the inductive hypothesis. \qedhere
  \end{enumerate}
\end{proof}

\begin{Lemma}\label{Reedy-levelwise}
  The inclusion $\N w(\c^{[k]}_\Reedy) \to \N w(\c^{[k]})$
  is a weak homotopy equivalence.
\end{Lemma}

\begin{proof}
  Functorial factorization induces a functor in the opposite direction
  as well as natural weak equivalences
  connecting both composites with identities.
\end{proof}

\begin{proof}[Proof of \cref{cC-C}]
  Recall that we want to show that $\N w((\cof\c)^{[k]}) \to \N w(\c^{[k]})$ is
  a weak equivalence for all $k$.
  In the diagram
  \begin{tikzeq*}
    \matrix[diagram]
    {
      & |(wcR)| \N w\cof(\c^{[k]}_\Reedy)
      & |(wR)|  \N w(\c^{[k]}_\Reedy) \\
        |(w-c)| \N w((\cof\c)^{[k]})
      & |(wc)|  \N w\cof(\c^{[k]})
      & |(w)|   \N w(\c^{[k]}) \\
    };

    \draw[->] (wcR) to node[above] {\ding{172}} (wR);
    \draw[->] (wc)  to node[below] {\ding{173}} (w);
    \draw[->] (w-c) to (wc);

    \draw[->] (wcR) to node[above left] {\ding{174}} (w-c);
    \draw[->] (wcR) to (wc);
    \draw[->] (wR)  to node[right] {\ding{175}} (w);
  \end{tikzeq*}
  the indicated maps are weak equivalences.
  The map \ding{172} is a weak equivalence by \cref{W-diagonal} applied to
  $\c^{[k]}_\Reedy$ with itself as a good subcategory and
  so is \ding{173} by the same argument applied to $\c^{[k]}$.
  The map \ding{174} is a weak equivalence by \cref{W-diagonal} applied to
  $\c^{[k]}_\Reedy$ with the good subcategory of levelwise cofibrations,
  which is indeed good by \cref{levelwise-good}.
  Finally, \ding{175} is a weak equivalence by \cref{Reedy-levelwise}.
  Hence by 2-out-of-3, the bottom composite is also a weak equivalence
  as required.
\end{proof}

\begin{ackn}
The first author was supported by Wolfgang L\"ucks ERC Advanced Grant ``KL2MG-interactions'' (no.662400) granted by the European Research Council.
\end{ackn}

\bibliographystyle{amsalpha}
\bibliography{mybib-1}

\providecommand{\bysame}{\leavevmode\hbox to3em{\hrulefill}\thinspace}
\providecommand{\MR}{\relax\ifhmode\unskip\space\fi MR }
\providecommand{\MRhref}[2]{%
  \href{http://www.ams.org/mathscinet-getitem?mr=#1}{#2}
}
\providecommand{\href}[2]{#2}
\begin{thebibliography}{CGT06}

\bibitem[Bro73]{Brown}
Kenneth~S. Brown, \emph{Abstract homotopy theory and generalized sheaf
  cohomology}, Trans. Amer. Math. Soc. \textbf{186} (1973), 419--458.
  \MR{0341469}

\bibitem[CGT06]{Casacuberta}
Carles Casacuberta, Marek Golasi{{\'n}}ski, and Andrew Tonks, \emph{Homotopy
  localization of groupoids}, Forum Math. \textbf{18} (2006), no.~6, 967--982.
  \MR{2278610 (2007j:18003)}

\bibitem[Cis12]{Cisinksi2}
Denis-Charles Cisinski, \emph{Does the classification diagram localize a
  category with weak equivalences},
  http://mathoverflow.net/questions/92916/does-the-classification-diagram-localize-a-category-with-weak-equivalences,
  April 2012.

\bibitem[Del12]{DellAmbrogio2}
Ivo Dell'Ambrogio, \emph{The unitary symmetric monoidal model category of small
  {$\rm C^*$}-categories}, Homology Homotopy Appl. \textbf{14} (2012), no.~2,
  101--127. \MR{3007088}

\bibitem[DL98]{DavisLueck}
Jim Davis and Wolfgang L\"uck, \emph{Spaces over a {C}ategory, {A}ssembly
  {M}aps, and {I}somorphism {C}onjectures in {K}- and {L}-{T}heory}, K-Theory
  \textbf{15} (1998), 201--251.

\bibitem[Dug01]{MR1870516}
Daniel Dugger, \emph{Combinatorial model categories have presentations}, Adv.
  Math. \textbf{164} (2001), no.~1, 177--201. \MR{1870516}

\bibitem[GJ99]{Goerss-Jardine}
Paul~G. Goerss and John~F. Jardine, \emph{Simplicial homotopy theory}, Progress
  in Mathematics, vol. 174, Birkh\"auser Verlag, Basel, 1999. \MR{1711612}

\bibitem[Hov99]{Hovey}
Mark Hovey, \emph{Model categories}, Mathematical Surveys and Monographs,
  vol.~63, American Mathematical Society, Providence, RI, 1999. \MR{1650134}

\bibitem[Joa03]{Joachim2}
Michael Joachim, \emph{{$K$}-homology of {$C^\ast$}-categories and symmetric
  spectra representing {$K$}-homology}, Math. Ann. \textbf{327} (2003), no.~4,
  641--670. \MR{2023312}

\bibitem[LN16]{LandNikolaus}
Markus Land and Thomas Nikolaus, \emph{On the relation between {K}- and
  {L}-theory of {C*}-algebras}, Arxiv:1608.02903 (2016).

\bibitem[Lur09]{LurieHTT}
Jacob Lurie, \emph{{H}igher {T}opos {T}heory}, Annals of Mathematics Studies,
  vol. 170, Princeton University Press, Princeton, NJ, 2009. \MR{2522659
  (2010j:18001)}

\bibitem[Lur14]{LurieHA}
\bysame, \emph{{H}igher {A}lgebra}, http://www.math.harvard.edu/~lurie/, 2014.

\bibitem[MG15]{Mazel-Gee}
Aaron Mazel-Gee, \emph{The universality of the {R}ezk nerve}, Arxiv:1510.03150
  (2015).

\bibitem[RB09]{Radulescu-Banu}
Andrei Radulescu-Banu, \emph{Cofibrations in homotopy theory},
  no.~arXiv:math/0610009v4.

\bibitem[Rez01]{Rezk}
Charles Rezk, \emph{A model for the homotopy theory of homotopy theory}, Trans.
  Amer. Math. Soc. \textbf{353} (2001), no.~3, 973--1007 (electronic).
  \MR{1804411}

\bibitem[RW98]{Raeburn}
Iain Raeburn and Dana~P. Williams, \emph{Morita equivalence and
  continuous-trace {$C^*$}-algebras}, Mathematical Surveys and Monographs,
  vol.~60, American Mathematical Society, Providence, RI, 1998. \MR{1634408}

\end{thebibliography}

\end{document}